\documentclass[a4paper,12pt]{article}

\usepackage[table,dvipsnames,svgnames,x11names,hyperref]{xcolor}
\usepackage[utf8]{inputenc}

\usepackage{lmodern}
\usepackage[T1]{fontenc}

\usepackage[top=1in, bottom=1.25in, left=1.25in, right=1.25in]{geometry}
\usepackage{amsmath,amssymb,amsthm}

\usepackage{tikz}
\usepackage{array}
\usepackage{enumerate}
\usepackage{graphicx}
\usepackage{float}
\usepackage{caption}
\usepackage{subcaption}
\usepackage[colorlinks=true, allcolors=MidnightBlue]{hyperref}

\newtheorem{thm}{Theorem}

\newtheorem{lem}[thm]{Lemma}
\newtheorem{prop}[thm]{Proposition}
\newtheorem{conj}[thm]{Conjecture}

\theoremstyle{definition}

\theoremstyle{remark}

\newtheorem*{rem}{Remark}

\author{Sel\c{c}uk Kayacan
  \thanks{I would like to thank to Volkmar Welker for suggesting me to work on Hayashi's Conjecture. I would also like to thank to David Stanovský for an improvement of the statement and proof of Theorem~\ref{thm:prim}.}
}
\title{On a conjecture about profiles of\\ finite connected racks}
\date{}

\begin{document}

\maketitle

\small

\begin{center}
  Bahçeşehir University, Faculty of Engineering\\ and Natural Sciences,
  Istanbul, Turkey\\
  {\it e-mail:} \href{mailto:selcuk.kayacan@eng.bau.edu.tr}{selcuk.kayacan@eng.bau.edu.tr}
\end{center}

\begin{abstract}
  A rack is a set with a binary operation such that left multiplications are automorphisms of the set and a quandle is a rack satisfying a certain condition. For a finite connected rack the cycle type of the permutation defined by left multiplication by an element is independent from the chosen element. This cycle type is called the profile of the rack. Hayashi conjectured, in the profile of a finite connected quandle, the length of a cycle must divide the length of the largest cycle. In this paper, we prove Hayashi's Conjecture in some particular cases.
  


\end{abstract}

\section{Introduction}

Racks and quandles are algebraic objects that are mostly studied in the context of knot theory. The defining properties of those objects are in a sense compatible with the Reidemeister moves which makes them useful to define knot invariants. For example, in \cite{Joy82} Joyce introduced \emph{knot quandle} and showed that it is a complete invariant for knots and in \cite{FR92} Fenn and Rourke proved that the \emph{fundamental rack} is a complete invariant for irreducible framed links in a $3$-manifold. There are various notational conventions used to define racks and quandles in literature. In this paper we define racks and quandles in the following way. 

A \emph{rack} $X$ is a set together with a binary operation $\triangleright\colon X\times X\to X$ satisfying the following two axioms:
\begin{itemize}
\item[\textbf{(A1)}] for all $x,y,z\in X$ we have $x \triangleright (y \triangleright z) = (x \triangleright y) \triangleright (x \triangleright z)$ 
\item[\textbf{(A2)}] for all $x,z\in X$ there is a unique $y\in X$ such that $x \triangleright y = z$
\end{itemize}
A rack $X$ is a \emph{quandle} if, additionally, it satisfies the following axiom:
\begin{itemize}
\item[\textbf{(A3)}] for all $x\in X$ we have $x\triangleright x = x$
\end{itemize}

Let $G$ be a group and for any two elements $a,b\in G$ let $a\triangleright b := aba^{-1}$. Then, the group $G$ as well as any conjugacy class of $G$ satisfies those three axioms; hence equipped with the conjugation operation $\triangleright$ each one of them is an example of a quandle. More generally, we call a subset $X$ of $G$ \emph{conjugation rack} if it is a rack with the conjugation operation inherited by the group.

In this paper we only consider finite racks, although some of the statements are also valid for infinite racks. It is customary to identify elements of a finite rack $X$ with integers $1,2,\dots,n$ if $X$ has $n$ elements. By axiom (A2), the map $\phi_x\colon y\to x\triangleright y$ is a permutation hence is an element of the symmetric group $S_{n}$ of degree $n$. A permutation $\sigma\in S_{n}$ is an automorphism of $X$ if and only if $\sigma(x\triangleright y) = \sigma(x)\triangleright \sigma(y)$ for every $x,y\in X$. By axiom (A1), the map $\phi_x$ is an automorphism of $X$.

The set of all automorphisms of $X$, denoted $\mathsf{Aut}(X)$, form a subgroup of $S_{n}$ which is called the \emph{automorphism group} of $X$. The map $\Phi\colon x\mapsto \phi_x$ defines a rack morphism from $X$ to $S_{n}$, where $S_{n}$ is considered as a quandle with conjugation operation (see Proposition~\ref{prop:act}). The \emph{inner automorphism group} of $X$, denoted $\mathsf{Inn}(X)$, is the subgroup of $\mathsf{Aut}(X)$ generated by the elements of $\Phi(X)$. Let $\sigma$ be an automorphism of $X$. Since $\sigma\phi_x\sigma^{-1} = \phi_{\sigma(x)}$ for every $x\in X$, we see that $\mathsf{Inn}(X)$ is a normal subgroup of $\mathsf{Aut}(X)$. In general $\mathsf{Inn}(X)$ and $\mathsf{Aut}(X)$ may be different groups.

A rack $X$ is \emph{faithful} if the map $\Phi\colon X\to \mathsf{Inn}(X)$ is injective. Suppose $X$ is faithful. Obviously, the map $\Phi$ is an isomorphism between $X$ and $\Phi(X)$, hence $X$ is a quandle. Moreover, the center of the inner automorphism group is trivial and the action of $\mathsf{Inn}(X)$ on $\Phi(X)$ by conjugation is faithful (see Proposition~\ref{prop:act}).

A rack $X$ is \emph{connected} if the action of its inner automorphism group on $X$ itself is transitive. Suppose $X$ is a connected rack. In that case, any two elements of $\Phi(X)$ can be conjugated by an element of $\mathsf{Inn}(X)$. Since in symmetric group conjugate elements have the same cycle type, connectedness of $X$ implies each element of $\Phi(X)$ has the same cycle type. This cycle type $\lambda = (\lambda_0^{a_0},\lambda_1^{a_1},\dots,\lambda_t^{a_t})$ is called the \emph{profile} of $X$. Here $\lambda_s$ $(0\leq s\leq t)$ are ordered increasingly. Notice that $x\triangleright x = x$ for every $x\in X$ if $X$ is a quandle which means in the profile $\lambda$ of a quandle $X$ the value of the least element $\lambda_0$ must always be $1$. The purpose of this paper is to prove some special cases of the following conjecture which was stated by Hayashi originally for quandles (see \cite[Conjecture~1.1]{Hay13}). 

\begin{conj}\label{conj:hay}
  For a finite connected rack $X$ with profile $\lambda = (\lambda_0^{a_0},\lambda_1^{a_1},\dots,\lambda_t^{a_t})$, each $\lambda_s$\, $(0\leq s \leq t)$ divides $\lambda_t$.
\end{conj}

Let $G$ be a finite group, $\alpha$ be an automorphism of $G$, and $H$ be a subgroup of $G$ pointwise fixed by $\alpha$, i.e., $\alpha(h) = h$ for every $h\in H$. One can easily verify that the coset space $G/H$ with the operation $xH\triangleright yH := x\alpha(x^{-1}y)H$ satisfies quandle axioms. It is called the \emph{homogeneous quandle} and denoted by $(G,H,\alpha)$. Any finite connected quandle is isomorphic to a homogeneous quandle. Therefore, given the complete list of transitive groups of degree $n$, it is possible to determine the complete list of non-isomorphic connected quandles of $n$ elements. This method is used by Vendramin to compute all connected quandles having less than $48$ elements. The list is available in \textsf{Rig}~\cite{rig}, a \textsf{GAP} package designed for computations related to racks and quandles. In Table~\ref{table:sq} we present permutations of a quandle which is the fourth quandle with 12 elements in Vendramin's list. The reader may easily verify that the following partition
$$  \{1,5,9\},\; \{2,6,10\},\; \{3,7,11\},\; \{4,8,12\}  $$
formed by taking the first three numbers in each row of the table is a block system for the action of the inner automorphism group on the rack itself. We shall prove that Conjecture~\ref{conj:hay} holds when the inner automorphism group acts on the rack primitively (see Theorem~\ref{thm:prim}).

\begin{table}[h!]
\centering
\caption{Permutations of SmallQuandle(12,4) in \textsf{Rig}}
\label{table:sq}
\begin{tabular}{rll}
  $\phi_1$ & :     &\quad $(1)(5,9)(2,4,3)(6,12,7,10,8,11)$   \\
  $\phi_2$ & :     &\quad $(2)(6,10)(1,3,4)(5,11,8,9,7,12)$   \\
  $\phi_3$ & :     &\quad $(3)(7,11)(1,4,2)(5,12,6,9,8,10)$   \\
  $\phi_4$ & :     &\quad $(4)(8,12)(1,2,3)(5,10,7,9,6,11)$   \\
  $\phi_5$ & :     &\quad $(5)(1,9)(6,8,7)(2,12,3,10,4,11)$   \\
  $\phi_6$ & :     &\quad $(6)(2,10)(5,7,8)(1,11,4,9,3,12)$   \\
  $\phi_7$ & :     &\quad $(7)(3,11)(5,8,6)(1,12,2,9,4,10)$   \\
  $\phi_8$ & :     &\quad $(8)(4,12)(5,6,7)(1,10,3,9,2,11)$   \\
  $\phi_9$ & :     &\quad $(9)(1,5)(10,12,11)(2,8,3,6,4,7)$   \\
  $\phi_{10}$ & :  &\quad $(10)(2,6)(9,11,12)(1,7,4,5,3,8)$   \\
  $\phi_{11}$ & :  &\quad $(11)(3,7)(9,12,10)(1,8,2,5,4,6)$   \\
  $\phi_{12}$ & :  &\quad $(12)(4,8)(9,10,11)(1,6,3,5,2,7)$   \\
\end{tabular}
\end{table}

Let $A$ be a finite abelian group and $\alpha$ be an automorphism of $A$. The corresponding homogeneous quandle $(A,1,\alpha)$ is called an \emph{affine quandle}. Let $\beta(x) := x - \alpha(x)$ for every $x\in A$. Observe that in  $(A,1,\alpha)$ the quandle operation is defined by
$$x\triangleright y := \alpha(y-x) + x = \alpha(y) + \beta(x) = \beta(x-y) + y$$
and the orbit of $y$ under the action of inner automorphism group is $y + \beta(A)$. Suppose  $(A,1,\alpha)$ is connected. Then $\beta$ must be a bijection of $A$. Moreover, the inner automorphism group is isomorphic to a cyclic extension of $A$, hence is solvable. Now, consider the profile of $(A,1,\alpha)$. Kajiwara and Nakayama proved in \cite{KN16} that there must be an element of $A$ whose orbit size under the action of $\langle\phi_x\rangle$ is the order of $\phi_x$. Therefore, Hayashi's Conjecture holds for connected affine quandles.

Let $C$ be a conjugacy class of a symmetric group $S_d$. Hence, the inner automorphism group of $C$ is either isomorphic to $A_d$ or $S_d$. In particular $\mathsf{Inn}(C)$ is not solvable if $d\geq 5$. Let $x$ be an element of $C$ and let $\ell$ be the order of $x$ in $S_d$. To prove Hayashi's Conjecture holds for $C$, one possible way is to find an element $y$ of $C$ so that $x^kyx^{-k}\neq y$ for $1\leq k < \ell$. However, this method seems to be not applicable for other classes of groups as we don't have an explicit description of the elements of conjugacy class $C$. We shall prove that Conjecture~\ref{conj:hay} holds for the conjugacy classes in a symmetric group using a similar approach (see Theorem~\ref{thm:sym}).

\section{Preliminaries}\label{sec:pre}

The reader may refer to \cite{AG03} for the basic theory of racks and in particular to Lemma~1.7, Lemma~1.9 and Proposition~3.2 of the same paper for alternative statements and proofs of the following Proposition.

\begin{prop}\label{prop:act}
  Let $X$ be a rack. Then following statements hold.
  \begin{enumerate}[(i)]
  \item The map $\Phi\colon X\to \mathsf{Inn}(X)$ defines a rack morphism between $X$ and $\Phi(X)$ which is an isomorphism if $X$ is faithful.
  \item If $X$ is faithful, then the action of $\mathsf{Inn}(X)$ on $X$ by automorphisms and on $\Phi(X)$ by conjugation are isomorphic and the center of $\mathsf{Inn}(X)$ is trivial.  
  \item If $X$ is a conjugation rack and $G := \langle X \rangle$, then the map $\Phi\colon X\to \Phi(X)$ extends to a map
  $$\Phi\colon G\to \mathsf{Inn}(X),$$
  which is a group homomorphism and $\mathsf{Inn}(X)\cong G/Z(G)$.
  \item The rack $X$ is connected and faithful if and only if it is isomorphic, as a quandle, to a conjugacy class $C$ of a group $G$ so that $G = \langle C\rangle$ and $Z(G) = 1$.
  \end{enumerate}

\end{prop}

\begin{proof}
  \emph{(i)}\; Observe that for every $x,y,z\in X$ the equality $\phi_x(y\triangleright z) = \phi_x(y)\triangleright \phi_x(z)$ holds by axiom (A1). In other words  $\phi_x\phi_y = \phi_{x\triangleright y}\phi_x$ holds for any two elements $x,y\in X$. However, this means the map $\Phi$ takes $x\triangleright y$ into $\phi_x\phi_y\phi_x^{-1} = \phi_x\triangleright\phi_y$ and so $\Phi$ is a rack morphism between $X$ and $\Phi(X)$. Further if $\Phi$ is injective then $X$ and $\Phi(X)$ would be isomorphic quandles.
  
  \emph{(ii)}\; Suppose $X$ is faithful. Pick an automorphism $\sigma\in \mathsf{Inn}(X)$. Since, for any $x\in X$,
  $$ \Phi(\sigma\cdot x) = \Phi(\sigma(x)) = \phi_{\sigma(x)} = \sigma\phi_x\sigma^{-1} = \sigma\cdot \Phi(x), $$
  we see that the map $\Phi$ defines an isomorphism between the actions of $\mathsf{Inn}(X)$ on $X$ and on $\Phi(X)$. Now, if $\sigma$ lies in the center of $\mathsf{Inn}(X)$, then $\phi_{\sigma(x)} = \sigma\phi_x\sigma^{-1} = \phi_x$ for every $x\in X$ implying $\sigma$ is the identity map on $X$ since the map $\Phi\colon x\mapsto \phi_x$ is a bijection.

  \emph{(iii)}\; Consider the map $\Phi\colon G\to \mathsf{Inn}(X)$ taking an element $g:=y_j\dots y_2y_1$ of $G$, written as the product of some elements $y_i\in X$ for $1\leq i\leq j$, to an element $\phi_g$ of the inner automorphism group $\mathsf{Inn}(X)$, where $\phi_g$ is defined as the composition
  $$\phi_g := \phi_{y_j}\dots\phi_{y_2}\phi_{y_1}.$$
  Since, for any element $x\in X$ we have
  $$ gxg^{-1} = (y_j\triangleright\dots(y_2\triangleright(y_1\triangleright x))\dots) = \phi_{y_j}\dots\phi_{y_2}\phi_{y_1}(x), $$
  the map $\Phi$ is well-defined and defines a surjective group homomorphism. Clearly, $g\in \mathsf{Ker}(\Phi)$ if and only if $\phi_g$ is the identity map which is the case if and only if $gxg^{-1} = x$ for every generator element $x \in X$.

  \emph{(iv)}\; Suppose $X$ is faithful and connected. Since $X$ is faithful, the rack $X$ is isomorphic to the conjugation rack $\Phi(X)$ and since $X$ is connected, conjugation rack $\Phi(X)$ is a conjugacy class of the inner automorphism group $\mathsf{Inn}(X)$. By the previous part the center of $\mathsf{Inn}(X)$ is trivial.

  Conversely, suppose $X$ is isomorphic to a conjugacy class $C$ of a group $G$ so that $G = \langle C\rangle$ and $Z(G) = 1$. By the previous part, the rack $C$ is connected and faithful. Therefore $X$ is faithful and connected as well.
\end{proof}

\begin{lem}[see {\cite[Lemma~1.21]{AG03}}]\label{lem:fiber}
    Let $X$ and $Y$ be two racks and let $f\colon X\to Y$ be a surjective rack homomorphism. Then, for any $y_1,y_2\in Y$, the cardinalities of the fibers $f^{-1}(y_2)$ and $f^{-1}(y_1\triangleright y_2)$ are equal. In particular, every fiber of $Y$ has the same cardinality if $Y$ is connected.
\end{lem}

\begin{proof}
  We show that there exist a bijective function $B\colon X\to X$ such that
  $$ B(f^{-1}(y_2)) \subseteq f^{-1}(y_1\triangleright y_2) \text{ and } B(f^{-1}(y_1\triangleright y_2)) \subseteq f^{-1}(y_2).$$
  Let $B(x) := a\triangleright x$, where $f(a) = y_1$. Observe that $B^{-1} = \phi_a^{-1}\in \mathsf{Inn}(X)$. Suppose $x\in f^{-1}(y_2)$. Then $f(a\triangleright x) = f(a) \triangleright f(x) = y_1\triangleright y_2$. That is, $B(x)\in f^{-1}(y_1\triangleright y_2)$. Next, suppose $x\in f^{-1}(y_1\triangleright y_2)$. Then $f(B^{-1}(x)) = f\phi_a^{-1}(x) = \phi_{f(a)}^{-1}f(x) = \phi_{y_1}^{-1}(y_1\triangleright y_2) = y_2$. That is, $B^{-1}(x)\in f^{-1}(y_2)$.
\end{proof}

\begin{rem}
Let $X$ be a rack and take two elements $x,y$ of $X$. Let $\lambda_y^x$ be the orbit size of $y$ and $\bar{\lambda}_y^x$ be the orbit size of $\phi_y$ under the actions of $\langle\phi_x\rangle$. By Lemma~\ref{lem:fiber}, $\Phi^{-1}$ yields a partition of $X$ into a disjoint union of fibers of the same size. Observe that for any integer $k$ the element $\phi_x^k(y)$ of $X$ does not lie in the fiber $\Phi^{-1}(\phi_y)$ unless $\phi_x^k\phi_y\phi_x^{-k} = \phi_y$. Therefore $\bar{\lambda}_y^x$ divides $\lambda_y^x$.
\end{rem}
Let $X$ be a finite connected rack whose profile is $\lambda = (\lambda_0^{a_0},\lambda_1^{a_1},\dots,\lambda_t^{a_t})$ and let $k$ be an integer. For a permutation $\sigma$ of $X$, the \emph{$k$-part} of $\sigma$ is the set of elements of $X$ which appear in a cycle of length $k$ in the permutation $\sigma$. And the \emph{$\lambda_s$-part} of $X$ is the multiset formed by the $\lambda_s$-parts of elements of $\Phi(X)$. For example, if $X$ is the fourth quandle with 12 elements in Vendramin's list, using Table~\ref{table:sq}, we see that the $2$-part of $X$ is
$$ \{ 5,9,6,10,7,11,8,12,1,9,2,10,3,11,4,12,1,5,2,6,3,7,4,8 \}.  $$

\begin{lem}\label{lem:multi}
  Let $X$ be a finite connected rack  with profile $\lambda = (\lambda_0^{a_0},\lambda_1^{a_1},\dots,\lambda_t^{a_t})$. Then the number of occurrences of each element of $X$ in the $\lambda_s$-part of $X$ is same and equal to $a_s\lambda_s/f$ for each $\lambda_s$\, $(0\leq s \leq t)$, where $f$ is the cardinality of a fiber of $\Phi$.
\end{lem}

\begin{proof}
  Let $\sigma$ be an element of $\mathsf{Inn}(X)$. Conjugating each element of $\Phi(X)$ with $\sigma$ yields a permutation of $\Phi(X)$. Since $\sigma\phi_x\sigma^{-1} = \phi_{\sigma(x)}$, if $\sigma(u) = v$ for some distinct elements $u$ and $v$ of $X$ then the number of occurrences of $u$ and $v$ in the  $\lambda_s$-part of $X$ would be the same. The claim follows from the facts that $X$ is connected and the cardinality of $\Phi(X)$ is $n/f$, where $n$ is the number of elements in $X$.  
\end{proof}

Next, we shall review Conjecture~\ref{conj:hay} in group theoretical terms. Let $X$ be a finite connected rack and fix an element $x\in X$. Let $F:=\langle\phi_x\rangle$ be the subgroup of $G:=\mathsf{Inn}(X)$ generated by $\phi_x$ and $H:=C_G(\phi_x)$ be the centralizer of $\phi_x$ in $G$. Clearly, $F\leq H$. The action of $G$ on $\Phi(X)$ by conjugation is equivalent to the action of $G$ on the coset space $G/H$ by left multiplication. Let $\lambda = (\lambda_0^{a_0},\lambda_1^{a_1},\dots,\lambda_t^{a_t})$ be the profile of $\Phi(X)$. Then, the permutation $\phi_{\phi_x}$ would correspond to the permutation
$$ (H)(\phi_x\phi_{y_1} H,\dots,\phi_x^{\lambda_1}\phi_{y_1} H)\dots (\phi_x\phi_{y_q} H,\dots,\phi_x^{\lambda_s}\phi_{y_q} H)\dots (\phi_x\phi_{y_r} H,\dots,\phi_x^{\lambda_t}\phi_{y_r} H), $$
where $\phi_{y_q}$ $(0\leq q \leq r = \sum_{s=0}^ta_s\lambda_s)$ are some representative elements from the cycles of $\phi_{\phi_x}$. Let $q := (\sum_{i=0}^{s-1}a_i\lambda_i) + j$ for some $1\leq j \leq a_s$. It should be clear that
$$ \lambda_s = \frac{| F \phi_{y_q} H |}{|H|}. $$
More importantly, the subgroup $\langle \phi_x^{\lambda_s}\rangle$ of $F$ is also a subgroup of $\phi_{y_q}H\phi_{y_q}^{-1}$. 

\begin{prop}\label{prop:re}
  Let $X$ be a finite connected rack and $x$ be an element of $X$. Let $F$ be the subgroup of $\mathsf{Inn}(X)$ generated by $\phi_x$ and $H$ be the centralizer of $\phi_x$ in $\mathsf{Inn}(X)$. Let $\lambda = (\lambda_0^{a_0},\lambda_1^{a_1},\dots,\lambda_t^{a_t})$ be the profile of $X$. The following statements hold:
  \begin{enumerate}[(i)]
  \item If $F$ intersects $\phi_yH\phi_y^{-1}$ trivially for some $\phi_y\in \Phi(X)$, then $\lambda_s$ divides $\lambda_t$ for each  $0\leq s \leq t$.
  \item Conversely, if $\lambda_s$ divides $\lambda_t$ for each  $0\leq s \leq t$ and $X$ is faithful, then $F$ intersects one of the conjugates of $H$ trivially.
  \end{enumerate}
\end{prop}

\begin{proof}
  \emph{(i)}\; Let $k$ be an integer and $\bar{\lambda} = (\bar{\lambda}_0^{a_0},\bar{\lambda}_1^{a_1},\dots,\bar{\lambda}_{t'}^{a_{t'}})$ be the profile of $\Phi(X)$. By the assumption
$$ (\phi_x\phi_{y} H,\phi_x^2\phi_{y} H,\dots,\phi_x^{k}\phi_{y} H)  $$
is an orbit under the action of $F$ if and only if $k$ equals to the order of $\phi_x$. Then, clearly, $k$ equals to $\bar{\lambda}_{t'}$. From the remark following Lemma~\ref{lem:fiber} we see that $k$ divides $\lambda_s$ for some $0\leq s \leq t$ so that the order of $\phi_x$ is equal to $\lambda_t$.
  
  \emph{(ii)}\; Since $\lambda_s$ divides $\lambda_t$ for each  $0\leq s \leq t$, we see that the order of $\phi_x$ is $\lambda_t$. By Proposition~\ref{prop:act} we know that $X$ and $\Phi(X)$ have the same profiles. Therefore, the order of $\phi_{\phi_x}$ is $\lambda_t$ as well. From this we conclude that there exists an element $\phi_y$ of $\Phi(X)$ so that $\phi_yH$ lies in an orbit of size $\lambda_t$ under the action of $F$, which means the intersection of $F$ with the centralizer of $\phi_{\phi_y(x)}$ is trivial.
\end{proof}

\section{Main results}\label{sec:main}

Let $G$ be a group acting transitively on a set $\Omega$. A partition of $\Omega$ form a block system if the following property is satisfied: for any element $g$ of $G$ an element of the partition is either mapped to itself by $g$ or to another element of the partition. In that case we call the elements of the partition \emph{blocks}. A block system is trivial if either the partition consists of the whole set or it consists of the singletons. We say $G$ acts \emph{primitively} on $\Omega$ if the only block systems are the trivial block systems.

Let $X$ be a finite connected rack with profile $\lambda = (\lambda_0^{a_0},\lambda_1^{a_1},\dots,\lambda_t^{a_t})$ and $k$ be an integer. For a permutation $\sigma$ of $X$, the \emph{$\tilde{k}$-part} of $\sigma$ is the set of elements of $X$ which appear in a cycle whose length is a divisor of $k$ in the permutation $\sigma$. Let $x,y$ be two elements of $X$. If an automorphism $\sigma$ of $X$ centralizes $\phi_y$, then as a permutation $\sigma$ preserves $\lambda_s$-part of $\phi_y$ for each $\lambda_s$\, $(0\leq s \leq t)$. Let $K$ be the $\tilde{k}$-part of $\phi_x$ and $L:=X\setminus K$ be the complement of $K$ in $X$. If $\phi_x^k$ centralizes $\phi_y$, then $\phi_y$ centralizes $\phi_x^k$ as well. However, that means as a permutation $\phi_y$ preserves the set $K$ as well as the set $L$. Consider again the fourth quandle with 12 elements in Vendramin’s list (see Table~\ref{table:sq}). Its profile is $\lambda = (1^1,2^1,3^1,6^1)$ so that $\tilde{2}$-part of $\phi_i$ has $3$ elements whereas $\tilde{3}$-part has $4$ elements. Now, if $\phi_j$ centralizes $\phi_i^2$ then $\tilde{2}$-parts of $\phi_i$ and $\phi_j$ must be the same since the cycle type of $\phi_i^2$ is $(1^3,3^3)$. Moreover, for this example the set of $\tilde{2}$-parts of permutations $\phi_j$ form a block system for the action of the inner automorphism group. However, we don't know any general method for relating $\tilde{k}$-parts with a block system. 

Let $H$ be the stabilizer of an element $x\in \Omega$. Then, the action of $G$ on  $\Omega$ is primitive if and only if the stabilizer $H$ is a maximal subgroup of $G$. To see this, observe that if there exists a proper subgroup $K$ of $G$ that contains $H$ strictly, then the cosets of $K$ would be unions of some of the cosets of $H$ in $G$. However, this determines a non-trivial block system for $G/H$ under the action of $G$ by left multiplication.

\begin{thm}\label{thm:prim}
  Let $X$ be a finite connected rack with profile $\lambda = (\lambda_0^{a_0},\lambda_1^{a_1},\dots,\lambda_t^{a_t})$. If the inner automorphism group $\mathsf{Inn}(X)$ acts primitively on $X$, then $\lambda_s$ divides $\lambda_t$ for each $0\leq s \leq t$.
\end{thm}

\begin{proof}
  By assumption $\mathsf{Inn}(X)$ acts primitively on $X$. Since the fibers of $\Phi\colon X\to \Phi(X)$ form a partition of $X$ which is a block system for the action of $\mathsf{Inn}(X)$, the racks $X$ and $\Phi(X)$ are isomorphic and have the same profiles. Moreover, by Proposition~\ref{prop:act} the center of $\mathsf{Inn}(X)$ is trivial.
  
  To prove the Theorem suppose contrarily $\lambda_s$ does not divide $\lambda_t$ for some $0\leq s\leq t$. We want to derive a contradiction. Let us fix an element $x\in X$ and let $y$ be an element of $X$ different from $x$. Let $F:=\langle\phi_x\rangle$ and $F':=\langle\phi_y\rangle$. Also, let $H$ be the centralizer of $F$ in $G:=\mathsf{Inn}(X)$ and $H'$ be the centralizer of $F'$. Notice that $H'$ is a conjugate of $H$. Consider the intersection $I:= F\cap H'$. Using Proposition~\ref{prop:re}(i) contrapositively we see that $I$ is a non-trivial subgroup of $G$. Now $I$ centralizes both $F'$ and $H$, so it lies in the center of $\langle F', H\rangle$. Since $G$ acts on $X$ primitively by assumption, we see that $H$ is a maximal subgroup of $G$. Clearly, we can choose $y$ so that $F'$ is not contained in $H$. However that means $I$ lies in the center of $G$ which is a contradiction.
\end{proof}

Recall that in a symmetric group two elements are conjugate if and only if they share the same cycle type. Let $C$ be a conjugacy class in a symmetric group $S_d$. Recall that a conjugacy class generates a normal subgroup of the group. Therefore, if $d\geq 5$, as a quandle the inner automorphism group $\mathsf{Inn}(C)$ of $C$ is isomorphic to the symmetric group $S_d$ when $C$ is the conjugacy class of an odd permutation and it is isomorphic to the alternating group $A_d$ when $C$ is the conjugacy class of an even permutation different from the identity permutation. Observe that $C$ is faithful when $d\geq 5$ by Proposition~\ref{prop:act}, since the centers of $A_d$ and $S_d$ are trivial.

\begin{thm}\label{thm:sym}
  Let $C$ be a conjugacy class in a symmetric group $S_d$ which is connected as a quandle. If $\lambda = (\lambda_0^{a_0},\lambda_1^{a_1},\dots,\lambda_t^{a_t})$ is the profile of $C$, then $\lambda_s$ divides $\lambda_t$ for each $0\leq s \leq t$.
\end{thm}

\begin{proof}
  It is clear that the Theorem holds when $C$ is the conjugacy class of a central element. So we assume in the rest of the proof that $C$ is the conjugacy class of a non-central element.
  
  First, we prove the Theorem when the degree of $S_d$ is equal to three or four. If $d=3$, there are two possibilities. Either $C$ is the conjugacy class of $(1,2)$ or it is the conjugacy class of $(1,2,3)$. In the first case, the inner automorphism group is $S_3$ and since the center of $S_3$ is trivial, $C$ is faithful by Proposition~\ref{prop:act}. The Theorem holds in this case since the order of $(1,2)$ is a prime number. The latter possibility cannot occur as the conjugacy class of $(1,2,3)$ is not connected as a rack.

  If $d=4$, there are four possibilities. The conjugacy class of $(1,2)(3,4)$ generates an abelian subgroup so it is not connected. But in all other three cases the conjugacy class is connected as a quandle. Since the order of an element of the conjugacy class is a prime power in all those cases the Theorem holds when the degree of the symmetric group is four.

  Next, suppose $d\geq 5$.  From the previous explanations we know that $C$ is faithful and the inner automorphism group is isomorphic to either $A_d$ or $S_d$ when $d\geq 5$. Now, fix an element $x$ of $C$ and let $\mu = (\mu_0^{b_0},\mu_1^{b_1},\dots,\mu_r^{b_r})$ be the cycle type of $x$ in $S_d$. Moreover, since the map $\Phi\colon C\to \mathsf{Inn}(C)$ extends to a group isomorphism, using  Proposition~\ref{prop:re}(i), it is enough to show that $\langle x\rangle$ intersects the centralizer of some element $y$ of $C$ trivially. Suppose $r\geq 1$. Let $y$ be an element of $C$ having the following property: For any $0\leq s\leq r$ the $\mu_s$-parts of $x$ and $y$ are not coinciding. If the value of $r\geq 1$, the existence of $y$ in $S_d$ or $A_d$ is clear. Let $\ell$ be the order of $x$ and $k$ be an integer dividing $\ell$. Obviously, $x^k$ does not centralize $y$ unless $k=\ell$. Hence the Theorem holds when $C$ is the conjugacy class of an element containing at least two cycles of different lengths in its cycle decomposition. Next, suppose $r=0$. In that case $d = b_0\mu_0$ and the order of $x$ is $\mu_0$. In turn, the order of $\phi_x$ is $\mu_0$ as well and we may assume $\mu_0$ is neither $1$ nor a prime power since the Theorem holds trivially in those cases. Recall that the centralizer of $x$ in $S_d$ is isomorphic to $C_{\mu_0}\wr S_{b_0}$, where $C_{\mu_0}$ is the cyclic group of order $\mu_0$. Moreover, for some conjugate $y$ of $x$ the intersection of $\langle x\rangle$ with the centralizer of $y$ is trivial when $\mu_0\geq 6$. This completes the proof.
\end{proof}

The proof of Theorem~\ref{thm:sym} can be adapted to work when $C$ is the conjugacy class of an alternating group. Notice that a conjugacy class of a simple group is necessarily connected as a quandle. The only complication arises when a conjugacy class of $A_d$ is not a conjugacy class of $S_d$. Actually this may happen. Let $x$ be an even permutation of $S_n$ whose cycle type is $\mu = (\mu_0^{b_0},\mu_1^{b_1},\dots,\mu_r^{b_r})$. The conjugacy class of $x$ in $S_n$ splits into two different conjugacy classes in $A_n$ if $\mu_s$ is an odd number and $b_s=1$ for every $0\leq s\leq r$. Otherwise, the conjugacy class of $x$ in $S_n$ is same with the conjugacy class of $x$ in $A_n$. Even in the former case, fixing an element $x$ of $C$, we can show that there exist an element $y$ of $C$ so that the centralizer of $y$ intersects $\langle x\rangle$ trivially.


\end{document}